\documentclass[12pt]{amsart}
\usepackage{amssymb,amscd,amsmath}
\usepackage[all]{xy}
\newtheorem{thm}{Theorem}[section]

\newtheorem{cor}[thm]{Corollary}

\newtheorem{prop}[thm]{Proposition}

\newtheorem{prob}[thm]{Problem}
\newtheorem{lem}[thm]{Lemma}

\theoremstyle{remark}

\theoremstyle{definition}
\newtheorem{defn}[thm]{Definition}

\newcommand{\Alt}{{\raise 2pt\hbox{$\scriptstyle\bigwedge$}}}

\newcommand{\go}{\rightarrow}

\newcommand{\e}{\epsilon}

\begin{document}
\title{Probabilistically nilpotent groups}

\author{Aner Shalev}
\email{shalev@math.huji.ac.il}
\address{Einstein Institute of Mathematics\\
    Hebrew University \\
    Givat Ram, Jerusalem 91904\\
    Israel}

%\subjclass{Primary 20E26; Secondary 20P05}

\thanks{The author was partially supported by ERC advanced grant 247034, BSF grant 2008194,
ISF grant 1117/13 and the Vinik Chair of mathematics which he holds.}

\begin{abstract}
We show that, for a finitely generated residually finite group $\Gamma$,
the word $[x_1, \ldots, x_k]$ is a probabilistic identity of $\Gamma$ if and
only if $\Gamma$ is virtually nilpotent of class less than $k$.

Related results, generalizations and problems are also discussed.

\end{abstract}

\maketitle

%\newpage

\section{Introduction}

A well known result of Peter Neumann \cite{N} shows that a finite group $G$ in which
the probability that two random elements commute is at least $\e > 0$
is bounded-by-abelian-by-bounded; this means that there are normal subgroups
$N, K$ of $G$ such that $K \le N$, $N/K$ is abelian, and both $|G/N|$ and $|K|$
are bounded above by some function of $\e$.

The probability that two elements commute received considerable attention
over the years, see for instance  \cite{G}, \cite{J}, \cite{LP}, \cite{GR}, \cite{GS}, \cite{H}, \cite{NY}, \cite{E}.
However, the natural extension to longer commutators and the probability of them being $1$ remained
unexplored.
Here we shed some light on this problem, providing some results, characterizations
and directions for further investigations.

Neumann's result, as well as a similar result of Mann on groups with many involutions \cite{M1},
can be viewed in the wider context of the theory of word maps
(see for instance the survey paper \cite{S} and the references therein) and the notion of
probabilistic identities.

A word $w=w(x_1, \ldots , x_k)$ is an element of the free group $F_k$ on $x_1, \ldots , x_k$.
Given a group $G$, the word $w$ induces a word map $w_G: G^k \to G$ induced
by substitution. We denote the image of this word map by $w(G)$.

If $G$ is finite then $w$ induces a probability distribution
$P_{G,w}$ on $G$, given by
\[
P_{G,w}(g) = |w_G^{-1}(g)|/|G|^k.
\]
A similar distribution is defined on profinite groups $G$, using their normalized Haar measure.

This distribution has been studied extensively in recent years, with particular
emphasis on the case where $G$ is a finite simple group -- see \cite{DPSS}, \cite{GS},
\cite{LS1}, \cite{LS2}, \cite{LS3}, \cite{LS4}, \cite{B}.
Here we focus on general finite groups and residually finite groups,
and the proofs of our results do not use the Classification of finite simple group.

Let $\Gamma $ be a residually finite group.
Recall that a word $w \ne 1$ is said to be a {\em probabilistic identity} of
$\Gamma $ if there exists $\e > 0$ such that for each finite quotient $H = \Gamma/\Delta$
of $\Gamma $ we have $P_{H,w}(1) \ge \e$.
This amounts to saying that, in the profinite completion $G = \widehat{\Gamma }$ of $\Gamma $,
we have $P_{G,w}(1) > 0$. We shall also be interested in finite and profinite groups
$G$ in which $P_{G,w}(g) \ge \e > 0$ for some element $g \in G$.

While Neumann's result deals with the commutator word $[x_1,x_2]$ of length two, here we consider
commutator words of arbitrary length. Define inductively $w_1 = x_1$ and $w_{k+1} = [w_k,x_{k+1}]$.
Thus $w_k$ is the left normed commutator $[x_1, \ldots , x_k]$.

Our main result characterizes finitely generated residually finite groups in which
such a word is a probabilistic identity. Clearly, if $\Gamma$ has a nilpotent normal
subgroup of finite index $m$ and class $< k$, and $G = \widehat{\Gamma }$, then
$P_{G,w_k}(1) \ge m^{-k} > 0$, so $w_k$ is a probabilistic identity of $\Gamma$.
It turns out that the converse is also true.

\begin{thm}
\label{main}
Let $\Gamma$ be a finitely generated residually finite group, and let $k$ be a positive integer.
Then the word $[x_1, \ldots , x_k]$ is a probabilistic identity of $\Gamma$ if and only if
$\Gamma$ has a finite index normal subgroup $\Delta$ which is nilpotent of class less than $k$.
\end{thm}

In fact our proof shows more, namely: if $\Gamma$ (or its profinite completion) is generated by $d$ elements,
and for some fixed $\e > 0$ and every finite quotient $H$ of $\Gamma$ there exists $h \in H$ such that
$P_{H,w_k}(h) \ge \e$, then the index $|\Gamma:\Delta|$ of the nilpotent subgroup $\Delta$ above divides $n!^{n!^d}$, where $n = \lfloor k/\e \rfloor$.

Since a coset identity is a probabilistic identity we immediately obtain the following.

\begin{cor} Let $\Gamma$ be a finitely generated residually finite group, and let $k$ be a positive integer.
Suppose there exist a finite index subgroup $\Delta$ of $\Gamma$, and elements
$g_1, \ldots , g_k \in \Gamma$ satisfying $[g_1\Delta, \ldots , g_k\Delta] = 1$.
Then there exists a finite index subgroup $\Delta_0$ of $\Gamma$ such that $\gamma_k(\Delta_0) = 1$.
\end{cor}

Here and throughout this paper $\gamma_k(G)$ denotes the $k$th term of the lower
central series of a group $G$.
Theorem \ref{main} follows from an effective result on finite groups, which is of
independent interest. To state it we need some notation.

For groups $G$ and $H$ we define
\[
G(H) = \cap_{\phi:G \go H}  \ker \phi,
\]
namely, the intersection of all kernels of homomorphisms from $G$ to $H$.
Then $G(H) \lhd G$, $G/G(H) \le H^l$ for some $l$ (which is finite if $G$ is finite).
Moreover, $G/G(H)$ satisfies all the identities of $H$, namely it lies in the variety generated
by $H$.

The case $H = S_n$, the symmetric group of degree $n$, will play a role below.
For a positive integer $n$ define
\[
L(n) = lcm (1, 2, \ldots , n),
\]
where $lcm$ stands for the least common multiple.
Clearly, $L(n)$ is the exponent of $S_n$. It is well known (and easy to verify using the
Prime Number Theorem) that $L(n) = e^{(1+o(1))n}$. There are considerably shorter identities for $S_n$,
but the length of its shortest identity is still unknown.

\begin{thm}
\label{finite}
Fix a positive integer $k$ and a real number $\e > 0$.

Let $w_k = [x_1, \ldots, x_k]$. Let $G$ be a finite group and suppose that
for some $g \in G$ we have $P_{G,w_k}(g) \ge \e$. Set $n = \lfloor k/\e \rfloor$
and let $N = G(S_n)$. Then $N$ is nilpotent of class less than $k$.

Furthermore, $G/N \le S_n^l$ for some $l$, and it satisfies all the identities
of $S_n$. In particular, $G/N$ has exponent dividing $L(\lfloor k/\e \rfloor)$.
\end{thm}

The following is an immediate consequence for residually finite groups
which are not necessarily finitely generated.

\begin{cor}
\label{non-fg}
If $w_k$ is a probabilistic identity of a residually finite
group $\Gamma$, then $\Gamma$ is an extension of a nilpotent group of class
less than $k$ by a group of finite exponent.
\end{cor}

In \cite{LS3} the following problem is posed.

\begin{prob}
\label{open1}
Do all finitely generated residually finite groups $\Gamma$ which satisfy
a probabilistic identity $w$ satisfy an identity?
\end{prob}

This seems to be a rather challenging problem.
Till recently the only non-trivial cases where a positive answer was known were
$w = [x_1,x_2]$ and $w = x_1^2$.

In \cite{LS3} an affirmative answer to Problem \ref{open1} is given for all
words $w$, provided the group $\Gamma$ is linear.

In \cite{LS4} it is shown that, if a residually finite group $\Gamma$ (not necessarily
finitely generated) satisfies a probabilistic identity,
then the non-abelian upper composition factors of $\Gamma$ have bounded size.
This leads to solutions of problems from \cite{DPSS} and \cite{B}.

The next result provides a positive answer to Problem \ref{open1} for additional words $w$;
it also deals with groups which are not finitely generated.

\begin{cor} Every residually finite group which satisfies the probabilistic
identity $w_k$ satisfies an identity.
\end{cor}

Indeed, Corollary \ref{non-fg} shows that our group satisfies the identity $[x_1^c, \ldots ,x_k^c]$
for some positive integer $c$.

\begin{defn} A word $1 \ne w \in F_k$ is said to be {\it good} if for any real number $\e > 0$
there exists a word $1 \ne v \in F_m$ (for some $m$) depending only on $w$ and $\e$ such that,
if $G$ is a finite group satisfying $P_{G, w}(g) \ge \e$ for some $g \in G$, then $v$ is an identity
of $G$.
\end{defn}

For example, results from \cite{N}, \cite{M1} and \cite{M2} imply that the words $[x_1,x_2]$ and $x_1^2$
are good. It is easy to see that, if $w$ is a good word, and $v$ is any word disjoint from $w$ (namely,
their sets of variables are disjoint), then the words $wv$ and $vw$ are also good.

Theorem \ref{finite} above shows that the words $[x_1, \ldots , x_k]$ are good for all $k$.
In fact, we can generalize the latter result as follows.

\begin{prop}\label{good} Let $w(x_1, \ldots , x_k)$ be a word, and let $w'(x_1, \ldots, x_{k+1}) =
[w(x_1, \ldots , x_k), x_{k+1}]$. Suppose $w$ is good. Then so is $w'$.
\end{prop}

Induction on $k$ immediately yields the following.

\begin{cor} The words $[x_1^2,x_2, \ldots, x_k]$ $(k \ge 2)$ are good.
\end{cor}

In fact we also obtain a related structure theorem as follows.

\begin{thm}\label{structure}

(i) Let $G$ be a finite group, let $w = [x_1^2,x_2, \ldots, x_k]$, and
suppose $P_{G,w}(g) \ge \e > 0$ for some $g \in G$. Then there exists $n = n(k,\e)$ depending only
on $k$ and $\e$, such that $G(S_n)$ is nilpotent of class at most $k$.

(ii) If $w$ above is a probabilistic identity of a finitely generated residually finite group
$\Gamma$, then $\Gamma$ has a finite index subgroup which is nilpotent of class at most $k$.
\end{thm}

\bigskip

\section{Proofs}

In this section we prove Proposition \ref{good}, Theorem \ref{finite}, Theorem \ref{main}
and Theorem \ref{structure}, which in turn imply the other results stated in the Introduction.

\begin{lem}\label{lem1} Let $w(x_1, \ldots , x_k)$ be a word, and let $w'(x_1, \ldots, x_{k+1}) =
[w(x_1, \ldots , x_k), x_{k+1}]$. Let $G$ be a finite group, and suppose $P_{G,w'}(g) \ge \e > 0$
for some $g \in G$.  Choose $g_1, \ldots , g_k \in G$ uniformly and independently.
Then, for every $0 < \delta < \e$,
\[
Prob(|G:C_G(w(g_1, \ldots , g_k))| < 1/\delta) > \e-\delta.
\]
\end{lem}

\begin{proof}

Choose $g_{k+1} \in G$ also uniformly and independently. Then
\[
Prob([w(g_1, \ldots , g_k),g_{k+1}] = g) = P_{G,w'}(g) \ge \e.
\]
Given $g_1, \ldots , g_k, g \in G$, the number of elements $g_{k+1} \in G$ satisfying
$[w(g_1, \ldots , g_k),g_{k+1}] = g$ is at most $|C_G(w(g_1, \ldots , g_k))|$.
This yields
\[
\e |G|^{k+1} \le \sum_{g_1, \ldots , g_k \in G} |C_G(w(g_1, \ldots , g_k))|.
\]
Let
\[
p = Prob(|C_G(w(g_1, \ldots , g_k))| > \delta |G|) = Prob(|G:C_G(w(g_1, \ldots , g_k))| < 1/\delta).
\]
Then we obtain
\[
\e \le p + (1-p)\delta,
\]
so $p \ge (\e-\delta)/(1-\delta) > \e - \delta$, as required.

\end{proof}

\begin{prop}\label{prop2} Let $G, k, w, w', \e$ be as in Lemma \ref{lem1}. Suppose $P_{G,w'}(g) \ge \e$
for some $g \in G$ and
let $0 < \delta < \e$. Let $N = G(S_n)$, where $n = \lfloor 1/\delta \rfloor$, and set
$M = C_G(N)$. Then $P_{G/M,w}(1) > \e - \delta$.
\end{prop}

\begin{proof}
Using Lemma \ref{lem1} we obtain
\[
Prob(|G:C_G(w(g_1, \ldots , g_k))| < 1/\delta) > \e - \delta.
\]
Clearly, if $|G:C_G(w(g_1, \ldots , g_k))| < 1/\delta$ then the permutation representation
of $G$ on the cosets of $C_G(w(g_1, \ldots , g_k))$ gives rise to a homomorphism
$\phi: G \to S_n$ satisfying $G(S_n) \le \ker(\phi) \le C_G(w(g_1, \ldots , g_k))$.
This implies that
\[
Prob(N \le C_G(w(g_1, \ldots , g_k)) > \e - \delta.
\]
Since $M = C_G(N)$ we obtain
\[
Prob(w(g_1, \ldots , g_k) \in M) > \e - \delta,
\]
so $P_{G/M,w}(1) > \e - \delta$, as required.
\end{proof}

\medskip

We now prove Proposition \ref{good}.

\begin{proof} Recall that $w \in F_k$ is a good word and $w' = [w,x_{k+1}]$.

To show that $w'$ is good, suppose $P_{G,w'}(g) \ge \e > 0$ for some $g \in G$.
Set $\delta = \e/2$, $n = \lfloor 1/\delta \rfloor = \lfloor 2/\e \rfloor$ and apply Proposition \ref{prop2}.
We obtain $P_{G/M,w}(1) > \e/2$, where $M = C_G(G(S_n))$. Since $w$ is good there is
a word $1 \ne v \in F_m$ depending on $w$ and $\e$ such that $v(G) \le M$.

Since $g^{L(n)} \in G(S_n)$ for all $g \in G$, it follows that $v' = [v,x_{m+1}^{L(n)}]$
is an identity of $G$, which depends on $w'$ and $\e$. Therefore $w'$ is good.

\end{proof}

Since $x_1$ is a good word, it now follows from Proposition \ref{good} by induction on $k$
that all commutator words $[x_1, \ldots , x_k]$ are good. We can now also prove the
more refined Theorem \ref{finite}.

\begin{proof}

We prove, by induction on $k \ge 1$, that, under the assumptions of the theorem,
for $n = \lfloor k/\e \rfloor$ and $N = G(S_n)$, we have $\gamma_k(N) = 1$.
The other statements of the theorem follow immediately.

If $k=1$ then $|G|^{-1} = P_{G,x_1}(g) \ge \e$ for some $g \in G$. This yields $|G| \le 1/\e$.
Let $n = \lfloor 1/\e \rfloor$. Then $N = G(S_n) = 1$, which yields the induction base.

Now, suppose the theorem holds for $k$ and we prove it for $k+1$.
We assume $P_{G,w_{k+1}}(g) \ge \e$ and let $n = \lfloor (k+1)/\e \rfloor$, $N = G(S_n)$
and $M = C_G(N)$.

Using Proposition \ref{prop2} with $w= w_k, w' = w_{k+1}$ and
$\delta = \e / (k+1)$ we obtain
\[
P_{G/M,w_k}(1) >  k\e/(k+1).
\]
By induction hypothesis this implies that $(G/M)(S_{\lfloor k/(k\e/(k+1)) \rfloor})$ is
nilpotent of class less than $k$. Since $k/((k\e/(k+1)) = (k+1)/\e$ we see that $(G/M)(S_n)$
is nilpotent of class less than $k$.
Clearly $(G/M)(S_n) \ge G(S_n)M/M$, and this yields
\[
\gamma_k(G(S_n)) \le M.
\]
Therefore
\[
\gamma_{k+1}(N) = [\gamma_k(N),N] \le [M,N]=1.
\]
This completes the proof.

\end{proof}

We can now prove Theorem \ref{main}. This result follows easily from Corollary \ref{non-fg}
using Zelmanov's solution to the Restricted Burnside Problem, which, for
general exponents, also relies on the Classification of finite simple groups.
However, we are able to provide an elementary proof of Theorem \ref{main} which avoids these
very deep results.

\begin{proof}

It suffices to show that, if $G$ is a $d$-generated finite group satisfying
$P_{G,w_k}(1) \ge \e$, then $G$ has a normal subgroup $N$ which is nilpotent
of class less than $k$, such that $|G/N|$ is bounded above in terms of $d, k , \e$
only.

Using Theorem \ref{finite} and its notation, $N = G(S_n)$ is nilpotent of class less than $k$.
We also have $G/N \le S_n^l$ for some $l$. Thus $G/N$ is a $d$-generated group lying in
the variety generated by $S_n$.

A classical result of B.H. Neumann \cite[14.3]{Ne} states that, for every finite group $H$ and a
positive integer $d$, the free $d$-generated group in the variety generated by $H$ is finite of
order dividing $|H|^{|H|^d}$.
This implies that
\[
|G/N| \le n!^{n!^d}.
\]
Since $n = \lfloor k/\e \rfloor$, $|G/N|$ is bounded above in terms of $d, k$ and $\e$.

This completes the proof of Theorem \ref{main}.

\end{proof}

Note that the conclusion of Theorem \ref{main} holds under the weaker assumption that,
for every finite quotient $H$ of $\Gamma$ there exists an element $h \in H$ such that
$P_{H,w_k}(h) \ge \e > 0$. Indeed this follows from Theorem \ref{finite} as above,
replacing $P_{G,w_k}(1)$ by $P_{G,w_k}(g)$.

Finally, we prove Theorem \ref{structure}.

\begin{proof}
We prove part (i) of the theorem by induction on $k$, starting with $k=1$
and $w = x_1^2$.

By Proposition 5 of \cite{M2}, if $P_{G,x_1^2}(g) \ge \e > 0$ then $P_{G, [x_1,x_2]}(1) \ge \e^2$.

By Theorem \ref{finite} this implies that $G(S_n)$ is abelian, where $n = \lfloor 2/\e^2 \rfloor$.

Now suppose the result holds for $k \ge 1$ and let us prove it for $k+1$.
Set $w = [x_1^2,x_2, \ldots, x_k]$ and $w' = [w,x_{k+1}]$.

We assume $P_{G,w'}(g) \ge \e > 0$. Choose $0 < \delta < \e$ (depending on $\e$), and set
$m = \lfloor 1/\delta \rfloor$
and $M = C_G(G(S_m))$. Then, by Proposition \ref{prop2} we have $P_{G/M,w}(1) > \e - \delta$.
By induction hypothesis there exists $n = n(k, \e-\delta)$ such that $(G/M)(S_n)$ is nilpotent of class
at most $k$. This yields
\[
\gamma_{k+1}(G(S_n)) \le M.
\]
Therefore
\[
[\gamma_{k+1}(G(S_n)),G(S_m)] \le [M,N] = 1.
\]
Define $n(k+1, \e) = \max (n(k, \e-\delta), \lfloor 1/\delta \rfloor)$. Then it follows that,
for $n'= n(k+1, \e) = \max (n,m)$ we have
\[
\gamma_{k+2}(G(S_{n'})) = 1,
\]
proving part (i).

Part (ii) follows from part (i) as in the proof of Theorem \ref{main}.

\end{proof}

By choosing $\delta$ to be a suitable function of $\e$ (so that $m=n$ at each inductive step) one may
obtain explicit good bounds on $n(k, \e)$. We leave this for the interested reader.

\bigskip

\section{Related problems}

We conclude with some natural questions and directions for further research.

\begin{prob}\label{open2} Characterize residually finite groups in which $[x_1, \ldots , x_k]$
is a probabilistic identity.
\end{prob}

In particular, are these groups virtually nilpotent of class less than $k$?

The answer is positive for $k=2$. Indeed, a result of L{\'e}vai and Pyber \cite[1.1(iii)]{LP}
shows that a profinite group $G$ in which $[x_1,x_2]$ is a probabilistic identity
has an open abelian subgroup, whose finite index need not be bounded in terms of
$P_{G,[x_1,x_2]}(1)$. This implies a similar result for residually finite groups.

By Theorem \ref{main}, the answer to the question above is positive for all $k$ if the ambient group $\Gamma$
(or its profinite completion) is finitely generated. In the general case it follows that, for some $n$, $\Gamma(S_n)$
(which may have infinite index in $\Gamma$) is nilpotent of class less than $k$.
But the reverse implication does not hold, as the product of infinitely many copies of
$S_n$ demonstrates.

\begin{prob}\label{open3} Are all non-identity words good?
\end{prob}

This does not seem likely (or provable), so it would be nice to find an
example of a non-identity word which is not good.

\begin{prob}\label{open4} Characterize the good words, or at least find more examples of them.
\end{prob}

This is particularly interesting for some specific words.

\begin{prob}\label{open5} Are power words $x_1^k$ good?
\end{prob}

Let us say that general commutator words are words constructed from the variables $x_k$ ($k \ge 1$)
in finitely many steps in which we pass from previously constructed words $w_1, w_2$ in disjoint sets of
variables to the word $[w_1,w_2]$.

For examples, let $\delta_1(x_1) = x_1$ and
\[
\delta_{k+1}(x_1, \ldots , x_{2^k}) = [\delta_k(x_1, \ldots , x_{2^{k-1}}), \delta_k(x_{2^{k-1}+1}, \ldots , x_{2^k})].
\]
Thus $G$ satisfies the identity $\delta_{k+1}$ if and only if it is solvable of derived
length at most $k$.

\begin{prob}\label{open6} Are general commutator words good? Are the words $\delta_k$ good?
\end{prob}

A positive answer would of course follow from a positive answer to the following.

\begin{prob}\label{open7} Suppose $w_1, w_2$ are good words in disjoint sets of variables.
Does it follow that the word $[w_1,w_2]$ is good?
\end{prob}

It would be nice to find analogues of Theorem \ref{main} where we
replace nilpotency by solvability.

\begin{prob}\label{open8} Let $\Gamma$ be a finitely generated residually finite group and suppose
the word $\delta_k$ is a probabilistic identity of $\Gamma$. Does it follow that
$\Gamma$ has a solvable subgroup $\Delta$ of finite index? Can we further require
that the derived length of $\Delta$ is less than $k$?
\end{prob}

Finally, for a finite group $G$, set
\[
Pr_k(G) = P_{G,w_k}(1),
\]
the probability that $[g_1, \ldots , g_k] = 1$ in $G$.
Note that $Pr_2(G)$, denoted in the literature by $Pr(G)$ and $cp(G)$, was widely studied
by various authors. It was shown in \cite{G} that the maximal value of $Pr(G)$ for $G$ non-abelian
is $5/8$.

\begin{prob} Given $k \ge 3$, find the maximal value of $Pr_k(G)$ for finite groups
$G$ satisfying $\gamma_k(G) \ne 1$.
\end{prob}

The set $\{ Pr(G): G \; {\rm a \; finite \; group} \}$ also received considerable attention;
see \cite{J}, \cite{H} and \cite{E}. In the latter paper, Eberhard shows that this set is well ordered by $>$
and that its limit points are all rational.

\begin{prob} For $k \ge 3$, study the set $\{ Pr_k(G): G \; {\rm a \; finite \; group} \}$.
Are its limit points all rational? Is it well ordered by $>$?
\end{prob}


\bigskip

\begin{thebibliography}{SGA3}


\bibitem[B]{B} A. Bors, Fibers of automorphic word maps and an application to composition factors,
{\it J. Group Theory} {\bf 20} (2017), 1103--1134.

\bibitem[DPSS]{DPSS} J.D. Dixon, L. Pyber, \'A. Seress, A. Shalev,
Residual properties of free groups and probabilistic methods,
{\it J. reine angew. Math. (Crelle's)} {\bf 556} (2003), 159--172.

\bibitem[E]{E} S. Eberhard, Commuting probabilities of finite groups,
{\it Bull. London Math. Soc.} {\bf 47} (2015), 796--808.

\bibitem[GS]{GS} S. Garion and A. Shalev, Commutator maps, measure preservation,
and $T$-systems, {\it Trans. Amer. Math. Soc.} {\bf 361} (2009),
4631--4651.

\bibitem[GR]{GR} R.M. Guralnick and G.R. Robinson, On the commutator probability in finite groups,
{\it J. Algebra} {\bf 300} (2006), 509--528; Addendum, {\it J. Algebra} {\bf 319} (2008), 18--22.

\bibitem[G]{G} W.H. Gustafson, What is the probability that two group elements commute?
{\it Amer. Math. Monthly} {\bf 80} (1973), 1031--1034.

\bibitem[H]{H} P. Hegarty, Limit points in the range of the commuting probability function on finite groups,
{\it J. Group Theory} {\bf 16} (2013), no. 2, 235--247.

\bibitem[J]{J} K. Joseph, Several conjectures on commutativity in algebraic structures, Amer. Math.
Monthly 84 (1977), 550--551.

\bibitem[LS1]{LS1} M. Larsen and A. Shalev,
Fibers of word maps and some applications,
{\it J. Algebra} {\bf 354} (2012), 36--48.

\bibitem[LS2]{LS2} M. Larsen and A. Shalev, On the distribution of values of
certain word maps, {\it Trans. Amer. Math. Soc.} {\bf 368} (2016),
1647--1661.

\bibitem[LS3]{LS3} M. Larsen and A. Shalev, A probabilistic Tits alternative
and probabilistic identities,
{\it Algebra and Number Theory} {\bf 10} (2016), 1359--1371.

\bibitem[LS4]{LS4} M. Larsen and A. Shalev, Words, Hausdorff dimension and randomly free groups,
{\it Math. Ann.}, to appear.

\bibitem[LP]{LP} L. L{\' e}vai and L. Pyber, Profinite groups with many commuting pairs or involutions,
{\it Arch. Math. (Basel)} {\bf 75} (2000), no. 1, 1--7.

\bibitem[M1]{M1} A. Mann, Finite groups containing many involutions,
{\it Proc. Amer. Math. Soc.} {\bf 122} (1994), 383--385.

\bibitem[M2]{M2} A. Mann, Groups satisfying identities with high probability, {\it Internat. J. Algebra Comput.},
to appear.

\bibitem[NY]{NY} R.K. Nath and M.K. Yadav, On the probability distribution associated
to commutator word map in finite groups, {\it Internat. J. Algebra Comput.} {\bf 25} (2015), no. 7,
1107--1124.

\bibitem[Ne]{Ne} B.H. Neumann, Identical relations in groups. I,
{\it Math. Ann.} {\bf 114} (1937), no. 1, 506--525.


\bibitem[N]{N} P.M. Neumann, Two combinatorial problems in
group theory, {\it Bull. London Math. Soc.} {\bf 21} (1989),
456--458.

\bibitem[S]{S} A. Shalev, Some problems and results in the theory
of word maps, {\it Erd\H{o}s Centennial}, eds: Lov\'{a}sz et al.,
Bolyai Soc. Math. Studies {\bf 25} (2013), pp. 611--649.




\end{thebibliography}
\end{document}